\documentclass[11pt]{amsart}
\usepackage{mypackage}
\usepackage{myformat}
\usepackage{hyperref}
\hypersetup{colorlinks=true,linkcolor=cyan}

\DeclareMathOperator{\intr}{Int}
\newcommand{\geH}{\succeq_{_\Ha}}
\newcommand{\Ha}{\mathcal H}

\newcommand{\Fa}{\mathcal F}
\newcommand{\Ga}{\mathcal G}
\newcommand{\Ev}[1]{\mathcal E(#1)}
\newcommand{\de}[1]{\delta_{#1}}
\newcommand{\br}[2]{[#1,#2]}
\newcommand{\bbr}[2]{\big[#1,#2\big]}
\newcommand{\Bbr}[2]{\Big[#1,#2\Big]}

\newcommand{\Print}[3][]{\Prob[#1]{#2;#3}}

\begin{document}
\mybic
\date \today 
 
\title
[Minimax]
{A Minimax Lemma and its Applications}

\subjclass[2020]{
Primary: 28C05, 
Secondary: 39B72, 39B62
.} 

\keywords{
Conglomerability,
Domination property,
Fan Theorem, 
Integral representation,
Integral hull,
Linear inequalities,
Minimax.
}

\begin{abstract} 
We prove an easy version of the minimax theorem 
with no topological assumption. We deduce from it
some domination criteria as well as an application
to $p$-summing operators.
\end{abstract}

\maketitle

\ToWhom{
This paper is dedicated to the memory of Riccardo Damasio.
}

\section{Introduction}

The main result of this paper is a minimax lemma which
we prove, in various degrees of generality, in section 
\ref{sec minimax}, using of a convexification technique 
based on finitely additive integration. In the subsequent 
sections we obtain applications to several, apparently 
independent problems which are rarely recognized as 
minimax problems. In section \ref{sec dom} we examine 
domination criteria for functions defined on arbitrary 
sets and obtain a non linear generalisation of a well
known Theorem of Ky Fan \cite{fan_56}. In section 
\ref{sec weak} we study a weaker form of domination 
similar to absolute continuity and involving convergence 
of functions to zero. In section \ref{sec choquet} we 
establish several integral representation theorems, 
extending the classical findings of Choquet \cite{choquet_56} 
and of Strassen \cite{strassen}. Eventually in section 
\ref{sec sum} we obtain applications to summable 
families of functions similar to the results of
Grothendieck and Pietsch. In all these applications 
the general structure of our minimax lemma often
permits extensions of the classical versions or just
alternative proofs.

As is well known, the problem of finding sufficient 
conditions for the minimax equality
\begin{equation}
\label{minimax}
\inf_{f\in\Fa}
\sup_{x\in X}
f(x)
	=
\sup_{x\in X}
\inf_{f\in\Fa}
f(x)
\end{equation}
originated in the theory of zero sum games with the 
classical work of von Neumann \cite{neumann} and
had immediate applications in several fields, e.g.
the theory of sequential statistical decisions of Wald
\cite{wald}. The abstract mathematical problem
received great impulse from the infinite dimensional 
generalisations obtained by Ky Fan \cite{fan_53,fan} 
and Maurice Sion \cite{sion}, later extended or 
improved by a number of other authors including
Ha \cite{ha}, 
Kindler \cite{kindler_93}, 
K\"onig \cite{konig}, 
Simons \cite{simons_94}
(who also discusses the different approaches) and 
Terkelsen \cite{terkelsen}. 
The conditions originally considered by Fan and Sion 
(and in some more general form also by much of the 
following literature) involve 
\tiref a
compactness of the space $X$,
\tiref b
some degree of convexity of $\Fa$ and concavity 
in $X$, 
\tiref c
some form of semicontinuity of the functions in $\Fa$. 
More recent contributions have replaced convexity 
or concavity with assumptions of a purely topological 
nature, such as connectedness (see e.g. \cite{horvath}, 
\cite{konig} and \cite{terkelsen}). 

The approach to the minimax problem we follow in this
work is based on the simple observation that even if, in 
the general case, the left hand side of \eqref{minimax} 
strictly exceeds the right hand one, we may still find a 
convenient extension $\widehat\Fa$ of the set $\Fa$ 
with the property that
\begin{equation}
\label{minimax Z}
\inf_{h\in\widehat\Fa}
\sup_{x\in X}
h(x)
	=
\sup_{x\in X}
\inf_{h\in\Fa}
h(x).
\end{equation}
In Theorem \ref{th minimax} we prove that one such 
extension is the integral hull, $\intr(\Fa)$, provided
that the family $\Fa$ is
\tiref i
pseudo concave on $X$ and 
\tiref{ii}
pointwise lower bounded. 
If we add to \tiref i -- \tiref{ii} the further assumption
that $\Fa$ is $B$-convex, a newly defined property
related to the existence of sub barycentres, then 
we recover the original minimax equality 
\eqref{minimax}. A comparison with the traditional 
assumptions of this literature, particularly 
compactness, shows that our result is indeed 
a generalization of those of Fan and Sion.

\subsection{Notation}

If $X$ and $Y$ are non empty  sets the symbol 
$\Fun{X,Y}$ (resp. $\Fun X$) denotes the family 
of all functions which map $X$ into $Y$ (resp. 
into $\R$). The topology of pointwise convergence
assigned to $\Fun X$ is referred to as the $X$-%
topology and the prefix $X$ will be used to mean
that a given class or operation is defined relative
to such topology. 
The set of all evaluations $e_x$ at some $x\in X$ 
will be denoted by $\Ev X$. The symbol $\Prob X$ 
(resp. $\Prob\A$) will designate the family of all 
finitely additive probabilities defined on the power 
set of $X$ (resp. defined on the algebra $\A$ of
subsets of $X$). As usual, $ba(\A)$ is the vector
space spanned by $\Prob\A$. If $m\in\Prob X$ 
and $f\in L^1(m)$, we shall use the symbols 
$\int fdm$ or $m(f)$ interchangeably. If 
$\Fa\subset\Fun X$, we write
\begin{equation}
\label{P(F)}
\Print X\Fa
	=
\left\{m\in\Prob X:\Fa\subset L^1(m)\right\}.
\end{equation}
Most often we shall be concerned with the set
$\Print\Fa{\Ev{X_0}}$ for some $\Fa\subset\Fun X$
and $X_0\subset X$.

\section{Barycentrical convexity}
\label{sec convex}

In minimax problems two properties are important: 
a form of boundedness and some extension of the notion 
of concavity/convexity for functions defined on an 
abstract set.

\begin{definition}
\label{def LB}
A family $\Fa\subset\Fun X$ is pointwise lower
bounded if
$\inf_{f\in\Fa}f(x)>-\infty$
for every $x\in X$.
\end{definition}

If $\Fa$ consists of the $Y$-sections of some function
$F\in\Fun{X\times Y}$, then pointwise lower boundedness
is implicit in the classical assumptions that $Y$ is 
compact and that the $X$-sections of $F$ are 
lower semicontinuous on $Y$.

Concerning convexity, we define two distinct notions which 
involve a sequence in $X$ and which, for this reason, we 
qualify as ``{\it pseudo}''. When the intervening sequence 
is replaced with a single point -- e.g. when $X$ is compact 
and $\Fa$ consists of lower semicontinuous functions -- this 
qualification is dropped. 

\begin{definition}
\label{def sion}
A family $\Fa\subset\Fun X$ is pseudo convex on $X$ if 
for all $x,x'\in X$ and $0\le t\le 1$ there exists a sequence 
$\seqn x$ in $X$ such that
\begin{equation}
\label{pseudo convex}
tf(x)+(1-t)f(x')+2^{-n}
	\ge
f(x_n)
\qquad
f\in\Fa,\ n\in\N.
\end{equation}
\end{definition}
Pseudo concavity is defined similarly. The definition 
of convexity for a family of functions on an abstract 
set $X$ is due to Fan \cite{fan_53} and was used by 
Sion \cite[2.2]{sion} under the name of convexlikeness
and by Le Cam \cite[Definition 7]{le_cam_64} who 
called it subconvexity. Pseudo convexity was introduced by Irle 
\cite[Theorem 3.1]{irle_85}. An intermediate property 
was examined by K\"onig \cite[Lemma]{konig_68} and 
by Terkelsen \cite[Theorem 2]{terkelsen} who assumed 
that for each pair $x,x'\in X$ there exists $x_0\in X$ 
such that
\begin{equation}
\label{terkelsen}
f(x)+f(x')
	\ge 
2f(x_0)
\qquad
f\in\Fa.
\end{equation}

\begin{definition}
\label{def barycentre}
(a)
A pseudo sub barycentre of $m\in\Print X \Fa$ on 
$\Fa\subset\Fun X$ is a sequence $\seqn x$ 
satisfying
\begin{equation}
\label{pseudo barycentre}
\int_Xf(x)m(dx)+2^{-n}
\ge
f(x_n)
\qquad
f\in\Fa,\ n\in\N.
\end{equation}
The set of probabilities admitting a pseudo 
sub barycentre on $\Fa$ is denoted by 
$\Print[\beta] X\Fa$.
(b)
$\Fa$ is pseudo $B$-convex on $X$ if each 
$m\in\Print X\Fa$ admits a pseudo sub barycentre 
on $\Fa$. 
\end{definition}

To understand the relation between the two
definitions above, observe that the point mass 
at each $x\in X$ is trivially in $\Print[\beta]X\Fa$
and that $\Fa$ is pseudo convex on $X$ if and 
only if every convex combination of point masses 
is itself an element of $\Print[\beta]X\Fa$. Thus, 
if $\Print[\beta]X\Fa$ is a convex set (e.g. if $\Fa$
is pseudo $B$-convex on $X$) then $\Fa$ is pseudo 
convex on $X$. On the other hand it is easily
deduced from Definitions \ref{def sion} and 
\ref{def barycentre} that if $\Fa$ is pseudo convex
then $\Print[\beta]X\Fa$ is a convex set.

Pseudo $B$-convexity of $\Fa$ on $X$ may be 
written as the condition that for each $\varepsilon>0$ 
and $m\in\Print X\Fa$ the family of sets 
$\{x\in X:f(x)-m(f)\le\varepsilon\}$ 
with $f\in\Fa$ has non empty intersection. This 
remark suggests an obvious link with compactness. 

\begin{lemma}
\label{lemma classical}
A family $\Fa\subset\Fun X$ which is pseudo convex 
on $X$ is also pseudo $B$-convex on $X$ in either 
one of the following special cases:
(a)
$\Fa$ is finite,
(b)
$X$ is a compact set and each $f\in\Fa$ is lower 
semicontinuous or
(c)
$\Fa$ is totally bounded in the metric of uniform 
distance on $X$.
\end{lemma}

\begin{proof}
We start with a useful, general fact (see \cite{rao} 
for details): for any $\mu\in\Prob X$,
$\Ha\subset L^1(\mu)$ finite and $\varepsilon>0$
there exist $x_1,\ldots,x_k\in X$ and convex
weights $\alpha_1,\ldots,\alpha_k$ such that
\begin{equation}
\label{fact}
\sup_{h\in\Ha}\Babs{
\int hd\mu-\sum_{j=1}^kh(x_j)\alpha_j
}
	<
\varepsilon.
\end{equation}
This may be seen by choosing $\eta>0$ in such a
way that, letting 
$B=\bigcap_{h\in\Fa}\{\abs h<\eta\}$,
\begin{equation}
\sup_{h\in\Ha}
\Babs{
\int hd\mu-\tfrac1{\mu(B)}\int_B hd\mu
}
	<
\varepsilon/2.
\end{equation}
Then, since $\Ha$ is uniformly bounded on $B$, we 
construct a finite partition $B_1,\ldots,B_k$ of $B$ 
such that
\begin{equation}
\sup_{h\in\Ha}\sup_{1\le j\le k}\sup_{x,x'\in B_j}
\abs{h(x)-h(x')}
	<
\varepsilon/2.
\end{equation}
We obtain \eqref{fact} by selecting $x_j\in B_j$
arbitrarily and setting $\alpha_j=m(B_j)/m(B)$.

Consider now $m\in\Print X\Fa$ and $\varepsilon>0$
as given.
\tiref a
If $\Fa$ is finite, we obtain from \eqref{fact}
\begin{align}
\label{fi}
\int fdm
	\ge
-\varepsilon+\sum_{j=1}^kf(x_j)\alpha_j
	\ge
-2\varepsilon+f(x_\varepsilon)
\qquad
f\in\Fa
\end{align}
in which the existence of $x_\varepsilon$ follows from 
pseudo convexity of $\Fa$. 
\tiref b
If each $f\in\Fa$ 
is lower semicontinuous and $X$ is compact, the 
sets of the form $\{f-m(f)\le\varepsilon\}$, with 
$\in\Fa$, are compact and have the finite intersection 
property, by \tiref a. Thus, \tiref b follows. To prove 
\tiref c, cover $\Fa$ with a finite number of disks of 
radius $\varepsilon$ with respect to the metric of 
uniform distance and let 
$f_1,\ldots,f_n\in\Fa$
be their centres. Then,
$\bigcap_{i=1}^n\{f_i-m(f_i)\le\varepsilon\}
	\subset
\bigcap_{f\in\Fa}\{f-m(f)\le3\varepsilon\}$
and the claim follows again from the finite intersection
property.
\end{proof}

Although related to one another, pseudo $B$-convexity 
and compactness are independent properties.

\begin{lemma}
\label{lemma bar}
Let $\Fa\subset\Fun X$ be pseudo convex on $X$.
$m\in\Print[\beta] X\Fa$ if and only if $m\in\Print X\Fa$ 
and the set
\begin{equation}
\label{NFL}
\Ha(m)
=
\Bcco[X]{\bigcup_{f\in\Fa}\{h\in\Fun X:h\le f-m(f)\}}
\end{equation}
does not contain positive constant functions (with
$\cco[X]\cdot$ indicating the $X$-closed convex 
hull).
\end{lemma}

\begin{proof}
If $m$ admits a pseudo sub barycentre on $\Fa$, then 
\begin{align*}
\sup_{h\in\Ha(m)}\inf_xh(x)
	\le
\sup_{f\in\Fa}\inf_xf(x)-m(f)
	\le
0
\end{align*}
so that  $\Ha(m)$ contains no positive constants. 
Conversely, fix $\varepsilon>0$ and let $\phi$ be 
a $X$-continuous linear functional that separates 
$\{\varepsilon\set X\}$ from $\Ha(m)$. It is easily 
seen that $\phi$ admits the representation as
$\phi(h)=\sum_{i=1}^ka_ih(x_i)$ 
for given
$(a_1,x_1),\ldots,(a_k,x_k)\in\R\times X$. 
We conclude that
\begin{equation}
\label{sepp}
\sup_{f\in\Fa,\ b\in\Fun{X,\R_+}}\sum_{i=1}^ka_i[f(x_i)-m(f)-b(x_i)]
	<
\varepsilon\sum_{i=1}^ka_i.
\end{equation}
The fact that $\Fun{X,\R_+}$ is a convex cone implies
$a_i\ge0$ for $i=1,\ldots,k$, the strict inequality
in \eqref{sepp} requires $\sum_{i=1}^ka_i>0$.
Let $\alpha_i=a_i/\sum_{i=1}^ka_i$. If $\Fa$
is pseudo convex on $Y$ then there exists 
$x^m_\varepsilon\in X$ 
satisfying
$
\varepsilon
	>
\sum_{i=1}^k\alpha_if(x_i)-m(f)
	\ge
f(x^m_\varepsilon)-m(f)
$
for each $f\in\Fa$. 
\end{proof}

\section{Main Theorem}
\label{sec minimax}

The preceding properties deliver an elementary 
version of the minimax lemma. 

\begin{theorem}
\label{th minimax}
Let $\Fa\subset\Fun X$ be pointwise lower bounded
and pseudo concave on $X$. Then,
\begin{equation}
\label{mm int}
\inf_{m\in\Print \Fa{\Ev X}}\sup_{x\in X }
\int_\Fa f(x)m(df)
	=
\sup_{x\in X}\inf_{f\in\Fa}f(x)
\end{equation}
and the infimum over $\Print \Fa{\Ev X}$ is attained.
If, in addition, $\Ev X$ is pseudo $B$-convex on $\Fa$,
then
\begin{equation}
\label{mm bary}
\inf_{f\in\Fa}\sup_{x\in X }
f(x)
	=
\sup_{x\in X }\inf_{f\in\Fa}f(x).
\end{equation}
\end{theorem}

\begin{proof}
Write
$\eta
	=
\sup_{x\in X}\inf_{f\in\Fa}f(x)$,
for brevity. Observe that, for any $m\in\Print\Fa{\Ev X}$,
\begin{equation}
\label{intinf}
\sup_{x\in X}\int_\Fa f(x)m(df)
	\ge
\eta
\end{equation}
so that the left hand side is always the largest between 
the two terms in \eqref{mm int}. It is thus enough to 
show that the converse of \eqref{intinf} holds for some 
$m\in\Print\Fa{\Ev X}$, a fact which is non trivial
only in the case $\eta<+\infty$ to which we shall limit 
attention. Form the convex cone $\mathcal K\subset\Fun \Fa$ 
spanned by the set 
$\{e_x-\eta:x\in X\}$. All elements of $\mathcal K$ are 
lower bounded functions while, by the definition of 
$\eta$ and the fact that $\Fa$ is pseudo concave in $X$, 
$\mathcal K$ admits no element $k\ge1$. It follows 
from \cite[Proposition 1]{JMAA_2009} that there 
exists 
$m_0
	\in
\Print\Fa{\Ev X}$ 
such that 
$\sup_{k\in\mathcal K}\int_\Fa k(f)m_0(df)
	\le
0$
i.e., in view of \eqref{intinf}, such that
\begin{equation}
\label{sep}
\sup_{x\in X}\inf_{f\in\Fa}f(x)
	\ge
\sup_{x\in X}\int_\Fa f(x)m_0(df)
=
\inf_{m\in\Print\Fa{\Ev X}}\sup_{x\in X}\int_\Fa f(x)m(df).
\end{equation}
Assume in addition that $m_0\in\Print[\beta]\Fa{\Ev X}$ 
and let $\seqn f$ be its pseudo sub barycentre on $\Ev X$.
Then,
$
2^{-n}+
\int_\Fa f(x)m_0(df)
	\ge
f_n(x)
$
for each
$x\in X$ and $n\in\N$ and consequently
\begin{align*}
2^{-n}+\sup_{x\in X}\inf_{f\in\Fa}f(x)
	\ge
\sup_{x\in X}f_n(x)
	\ge
\inf_{f\in\Fa}\sup_{x\in X}f(x)
\end{align*}
which proves the second claim.
\end{proof}

A special case of Theorem 
\ref{th minimax}, treated by Sion \cite{sion}, 
is that of a function $H\in\Fun{X\times Y}$ whose 
$Y$-sections $H_y$ are lower semicontinuous and 
form a concave family on $X$ and whose 
$X$-sections $H_x$ are upper semicontinuous and 
form a convex family on $Y$ and $Y$ is a compact
space.

The equality \eqref{mm int} is established in Theorem
\ref{th minimax} under minimal assumptions if one 
accepts to replace $\Fa$ with its {\it integral hull} 
defined as
\begin{equation}
\label{intr}
\intr(\Fa)
	=
\left\{
\int_\Fa f(\cdot)m(df):m\in\Print \Fa{\Ev X}
\right\}.
\end{equation}
The importance of the integral hull was clearly
understood by Dynkin \cite{dynkin} (his definition
is slightly different) who refers to a set $\Fa$
such that $\Fa=\intr(\Fa)$ as a convex measurable 
space.

All properties of $\Fa$ involving shape -- such as 
positivity, monotonicity and the like -- are preserved 
in passing from $\Fa$ to $\intr(\Fa)$. Moreover, 
if $X$ is a Banach space and $\Fa$ the unit sphere 
of its dual space then $\intr(\Fa)=\Fa$. On the other 
hand, properties involving limits, such as continuity, 
do not carry over unless they hold uniformly in $\Fa$: 
e.g. if $\Fa$ is equicontinuous then so is $\intr(\Fa)$.
Thus, a solution of a given problem that may be found 
in $\intr(\Fa)$ rather than $\Fa$ may still be 
acceptable in several instances. 

In geometric terms, it is clear that $\intr(\Fa)$ is
a convex set containing $\Fa$. On the other hand, 
it follows from \eqref{fact} that for given 
$\varepsilon>0$ and $X_0\subset X$ finite there 
exist points $f_1,\ldots,f_k\in\Fa$ and convex 
weights $\alpha_1,\ldots,\alpha_k$ such that 
\begin{align*}
\sup_{x\in X_0}\Babs{
\int_\Fa f(x)m(df)-\sum_{j=1}^kf_j(x)\alpha_j}
	<
\varepsilon.
\end{align*}
Thus $\intr(\Fa)\subset\cco[X]\Fa$ (the converse
inclusion requires the additional assumptions of
Corollary \ref{cor hull} below) from which we 
conclude:
\begin{lemma}
If $\Fa\subset\Fun X$ is pointwise lower bounded
and pseudo concave on $X$ then,
\begin{equation}
\label{mm cco}
\inf_{h\in\cco[X]\Fa}
\sup_{x\in X }
h(x)
	=
\sup_{x\in X }
\inf_{h\in\cco[X]\Fa}
h(x)
\end{equation}
\end{lemma}

We easily recover a local version of Theorem 
\ref{th minimax} similar to a result of Ha 
\cite[Theorem 4]{ha}.

\begin{corollary}
Let $\{\Fa_\alpha:\alpha\in\mathfrak A\}$ be a family
of subset of $\Fa\subset\Fun X$ each of which pointwise 
lower bounded and pseudo concave on $X$. Define
$\mathscr M_\alpha
	=
\{m\in\Print \Fa{\Ev X}:m(\Fa_\alpha^c)=0\}$
and
$\mathscr M
	=
\bigcup_{\alpha\in\mathfrak A}
\mathscr M$. 
Then,
\begin{equation}
\label{local minimax}
\inf_{\mu\in\mathscr M}
\sup_{x\in X}
\int_\Fa f(x)\mu(df)
	=
\inf_{\alpha\in\mathfrak A}
\sup_{x\in X}
\inf_{f\in\Fa_\alpha}
f(x)
\end{equation}
\end{corollary}

\begin{proof}
Fix $\alpha\in\mathfrak A$. Given that any 
$m\in\Print{\Fa_\alpha}{\Ev X}$ extends to some 
$\mu\in\mathscr M_\alpha$, by Theorem \ref{th minimax}
we have
\begin{align*}
\sup_{x\in X}
\inf_{f\in\Fa_\alpha}
f(x)
	=
\inf_{m\in\Print{\Fa_\alpha}{\Ev X}}
\sup_{x\in X}
\int_{\Fa_\alpha} f(x)m(df)
	=
\inf_{\mu\in\mathscr M_\alpha}
\sup_{x\in X}
\int_\Fa f(x)m(df)
\end{align*}
from which we easily obtain \eqref{local minimax}.
\end{proof}

For each $\mu\in\mathscr M$ the set $\Fa$ is 
$\mu$-a.s. pointwise lower bounded and pseudo 
concave.

A useful generalization of Theorem \ref{th minimax}
permits to drop concavity upon passing to the 
free vector space generated by $X$ (see e.g. 
\cite[p. 137]{lang}). This may be represented 
as the space $\Fun[0]X$ of all real valued 
functions on $X$ with finite support. Associating 
each $x\in X$ with the function $\de x\in\Fun[0]X$ 
which is $1$ at $x$ and $0$ elsewhere, is an 
embedding of $X$ into $\Fun[0]X$. We also 
notice that $\Fun[0]X$ is (isomorphic to) the 
dual space of $\Fun X$ relatively to the 
$X$-topology via the identity
\begin{equation}
\label{embed}
\br f h
	=
\sum_{x\in X}f(x)h(x)
\qquad
f\in\Fun X,\ h\in\Fun[0]X.
\end{equation}
More precisely the $X$-topology on $\Fun X$ coincides
with the weak topology induced by $\Fun[0]X$ via
\eqref{embed}.


\begin{theorem}
\label{th minimax F}
Let $\Ha\subset\Fun[0]{X,\R_+}$ be a convex set and let
$\Fa\subset\Fun X$ be pointwise lower bounded. Then,
\begin{equation}
\label{minimax F}
\min_{F\in\intr(\Fa)}\sup_{h\in\Ha}
\br Fh
	=
\sup_{h\in\Ha}\inf_{f\in\Fa}
\br fh.
\end{equation}
If $\Fa$ is pointwise bounded, then \eqref{minimax F} 
remains true upon replacing $\Ha$ with any convex 
subset of $\Fun[0]X$.
\end{theorem}

\begin{proof}
Under either assumption, $\Ha\subset\Fun[0]{X,\R_+}$ 
and $\Fa$ pointwise lower  bounded or $\Fa$ 
pointwise bounded, the collection of all functionals 
on $\Ha$ associated with some $f\in\Fa$ via 
\eqref{embed} is pointwise lower bounded. 
Moreover, the function $\br fh$ is concave on 
the convex set $\Ha$. We can then apply 
Theorem \ref{th minimax} and obtain 
\eqref{mm int}, of which \eqref{minimax F} is 
clearly an equivalent reformulation.
\end{proof}

\section{Strong domination properties}
\label{sec dom}

We may rewrite Theorem \ref{th minimax} to obtain 
a useful domination condition. It is convenient to 
adopt the symbol
\begin{equation}
\label{Delta}
\Delta(X)
	=
\{\delta\in\Fun[0]{X,\R_+}:\br{\set X}\delta=1\}.
\end{equation}

\begin{theorem}
\label{th dom}

Let $\Fa,\Ga\subset\Fun X$, with $\Fa$ pointwise 
upper bounded. The inequality
\begin{equation}
\label{balance}
\sum_{i=1}^n\br{g_i}{\delta_i}
	\le
\sup_{f\in\Fa}\Bbr f{\sum_{i=1}^n\delta_i}
\end{equation}
holds for every finite subset of $\Ga\times\Delta(X)$
if and only if there is $m\in\Print X{\Ev X}$
such that 
\begin{equation}
\label{dominate}
g(x)
	\le
\int_\Fa f(x)m(df),
\qquad
g\in\Ga,\ 
x\in X.
\end{equation}
\end{theorem}

\begin{proof}
Define the maps $G\in\Fun{\Ga\times X}$ and 
$T\in\bFun{\Fun X,\Fun{\Ga\times X}}$
implicitly by letting
\begin{equation}
G(g,x)=g(x)
\qand
(Tj)(g,x)=j(x)
\qquad
g\in\Ga,\ x\in X,\ j\in\Fun X.
\end{equation}
Write 
$\Ha
	=
\big\{h\in\Fun[0]{\Ga\times X}_+:
h(g,\cdot)\in\Delta(X)\text{ for all }g\in\Ga\big\}$. Then, 
\eqref{balance} takes the form
\begin{equation}
\sup_{h\in\Ha}
\ \inf_{f\in\Fa}
\bbr{ G-Tf}{ h}
	\le
0.
\end{equation} 
while the family  
$\big\{G-T f:f\in\Fa\big\}
	\subset
\Fun{\Ga\times X}$ 
is pointwise lower bounded. By Theorem 
\ref{th minimax}, this implies the inequality
\begin{equation}
0
\ge
\sup_{h\in\Ha}
\Bbr{\int_\Fa\big(G-T f\big)(\cdot)
m(d f)}
{h}
	=
\sup_{h\in\Ha}
\Bbr{G-\int_\Fa(T f)(\cdot)
m(d f)}
{h}
\end{equation} 
for some  
$m
	\in
\Print\Fa{\Ev{ X}{\Fa}}$. 
From this we deduce
\begin{equation}
g(x)
	\le
\int_{\Fa}(T f)(g,x) m(d f)
	=
\int_{\Fa} f(x) m(d f)
\qquad
g\in\Ga,\ x\in X.
\end{equation}
The converse implication is obvious.
\end{proof}

Several useful and known results follow easily from 
Theorem \ref{th dom}, upon assuming linearity. One 
special case is the domination Theorem of Ky Fan 
\cite[Theorem 12, p. 123]{fan_56} in which $X$ is a 
Banach space, $g\in\Fun X$ and $\Fa=\rho S_{X^*}$ 
(the ball of radius $\rho>0$ in the dual space $X^*$). 
Then, $\intr(\Fa)=\Fa$ and $g$ is dominated by a 
continuous linear functional with norm $\le\rho$ if 
and only if
\begin{equation}
\sum_{i=1}^Np_ig(x_i)
	\le
\rho\Bnorm{\sum_{i=1}^Np_ix_i}
\qquad
p_1,\ldots,p_N\in\R_+,\ \sum_{i=1}^Np_i\le1.
\end{equation}
As is well known, Fan's Theorem has been widely used 
in game theory to prove that the value of a game has 
non empty core, see \cite{delbaen}, and the condition 
corresponding to \eqref{balance} is known in that
literature as {\it balancedness}.

For each index $\alpha$ in some non empty set 
$\mathfrak A$ let $X_\alpha$ be a set, 
$g_\alpha\in\Fun{X_\alpha}$ and let $\pi_\alpha$ 
be the projection of $X=\bigtimes_\alpha X_\alpha$ 
on its $\alpha$-th coordinate. Letting
$\Ga
	=
\{g_\alpha\circ\pi_\alpha:\alpha\in\mathfrak A\}$
Theorem \ref{th dom} provides a simple necessary 
and sufficient criterion for the existence of common 
extensions with given marginals. This problem was 
treated by Strassen \cite{strassen} in a well known 
paper for the case of measures while a recent 
characterization for linear functionals was 
obtained by Berti and Rigo \cite{berti_rigo_21} (but 
see also \cite[Lemma 2]{berti_rigo_2004}).
The present version does not require linearity.

A last implication is the characterisation of the integral 
hull of a pointwise bounded family of functions.

\begin{corollary}
\label{cor hull}
If $\Fa\subset\Fun X$ is pointwise bounded then
$\intr(\Fa)=\cco[X]\Fa$.
\end{corollary}

\begin{proof}
We already noticed that
$\intr(\Fa)\subset\cco[X]\Fa$. If $g\in\cco[X]\Fa$,
apply Theorem \ref{th dom} with $\Ga=\{g,-g\}$
and replacing $\Fa$ with $\Fa\cup(-\Fa)$. Then
\eqref{balance} is true and, as a consequence, 
\eqref{dominate} holds with equality.
\end{proof}

\section{Weak domination properties}
\label{sec weak}

The pointwise domination criterion \eqref{dominate} 
is made less restrictive if we focus on the behaviour 
of functions in approaching (or attaining) zero. The 
characterization of this weaker criterion involves an 
appropriate decomposition of $X$.

If $\Fa\subset\Fun X$, define the set function
(recall \eqref{Delta})
\begin{equation}
I_\Fa(U)
	=
\inf_{h\in\Delta(U)}
\sup_{f\in\Fa}
\bbr{\abs f\wedge 1}h
\qquad
U\subset X.
\end{equation}
This function is instrumental to the following

\begin{definition}
\label{def cover}
Let $\kappa$ be a cardinal number and $\Fa\subset\Fun X$. 
We speak of a family 
$\{X_\alpha:\alpha\in\mathfrak A\}$
of subsets of $X$ as a $\kappa$-exhaustion
induced by $\Fa$ if:
(i)
$\mathfrak A$ has cardinality $\le\kappa$,
(ii)
$I_\Fa(X_\alpha)>0$ for every $\alpha\in\mathfrak A$
and 
(iii)
$\sup_{(x,f)\in X_0\times\Fa}\abs{f(x)}
	=
0$,
where
$X_0
	=
X\setminus\bigcup_{\alpha\in\mathfrak A}X_\alpha$.
A $\kappa$-exhaustion is residual if condition
(iii) is replaced with
(iii') 
$\lim_df(x_d)=0$ for each net $\net xdD$
which is eventually in $\bigcap_{i=1}^nX_{\alpha_i}^c$
for every $\alpha_1,\ldots,\alpha_n\in\mathfrak A$.
\end{definition}

\begin{theorem}
\label{th kelley}
Let $\kappa$ be an infinite cardinal number and
$\Fa\subset\F(X,[0,1])$. Then,
\begin{enumerate}[(a).]
\item
$\Fa$  induces a $\kappa$-exhaustion
of $X$ if and only if there exists 
$\Ga
	\subset
\intr(\Fa)$
of cardinality $\le\kappa$ and such that
\begin{equation}
\label{k kelley}
\sup_{g\in\Ga} g(x)=0
\qtext{implies}
\sup_{f\in\Fa} f(x)=0
\qquad
x\in X;
\end{equation}
\item
$\Fa$ induces a residual, $\kappa$-exhaustion of $X$
if and only if there exists 
$\Ga
	\subset
\intr(\Fa)$
of cardinality $\le\kappa$ and such that for every
net $\net xdD$ in $X$
\begin{equation}
\label{a kelley s}
\sup_{g\in\Ga}\lim_dg(x_d)
	=
0
\qtext{implies}
\sup_{f\in\Fa}\lim_df(x_d)
	=
0.
\end{equation}
\end{enumerate}
\end{theorem}

\begin{proof}
Assume that $\{X_\alpha:\alpha\in\mathfrak A\}$ 
is a $\kappa$-exhaustion of $X$ induced by $\Fa$ 
and fix $\alpha\in\mathfrak A$. It follows from 
Theorem \ref{th minimax F} that
\begin{equation}
\label{bu}
I_\Fa(X_\alpha)
	=
\sup_{F\in\intr(\Fa)}
\inf_{\gamma\in\Delta(X_\alpha)}
\br F h
	=
\sup_{F\in\intr(\Fa)}
\inf_{x\in X_\alpha}
F(x).
\end{equation}
From the assumption we infer the existence of 
$g_\alpha\in\intr(\Fa)$ such that 
$\inf_{x\in X_\alpha}g_\alpha(x)
	\ge
I_\Fa(X_\alpha)/2
	>
0$. Let 
$\Ga
	=
\{g_\alpha:\alpha\in\mathfrak A\}$.
The set $\Ga$ has cardinality $\le\kappa$;
moreover, if 
$\sup_{\alpha\in\mathfrak A}g_\alpha(x)
	=
0$
then necessarily $x\in X_0$ and so
$\sup_{f\in\Fa}f(x)
	=
0$,
as in \eqref{k kelley}. If, on the other hand, $\net xdD$ 
is a net such that $\lim_dg(x_d)=0$ for all $g\in\Ga$ 
then, no matter the choice of
$\alpha_1,\ldots,\alpha_k\in\mathfrak A$,
we must have $x_d\in\bigcap_{i=1}^kX_{\alpha_i}^c$
for all $d\in D$ sufficiently large and, if the exhaustion 
is residual, $\lim_df(x_d)=0$ for all $f\in\Fa$, as in 
\eqref{a kelley s}. This proves the direct implication 
of both claims, \tiref a and \tiref b.

Assume conversely that $\Ga\subset\intr(\Fa)$ has
cardinality $\le\kappa$ and satisfies \eqref{k kelley}, 
e.g. when \eqref{a kelley s} hods true. Let 
$\mathfrak A=\Ga\times\N$ and write each
element of $\mathfrak A$ as 
$\alpha=(g_\alpha,p_\alpha)$.
Define the sets
\begin{equation}
\label{exh}
X_\alpha
	=
\left\{x\in X:g_\alpha(x)>1/p_\alpha\right\}
\qquad
\alpha\in\mathfrak A
\end{equation} 
and
$X_0
	=
X\setminus\bigcup_{\alpha\in\mathfrak A} X_\alpha$.
Since $\kappa$ is infinite, the cardinality of $\mathfrak A$ 
does not exceed $\kappa$ and for each 
$\alpha\in\mathfrak A$ we have
\begin{equation}
I_\Fa(X_\alpha)
	=
\inf_{h\in\Delta(X_\alpha)}\sup_{f\in\Fa}\br fh
	\ge
\inf_{h\in\Delta(X_\alpha)}\br{g_\alpha}h
	\ge
1/p_\alpha.
\end{equation}
Moreover, if $x\in X_0$ we deduce that 
$\sup_\alpha g_\alpha(x)=0$
so that 
$\sup_{f\in\Fa}f(x)=0$,
by \eqref{k kelley}. Thus $\{X_\alpha:\alpha\in\mathfrak A\}$
is a $\kappa$-exhaustion. Eventually, let $\net xdD$ be a 
net in $X$ and fix $g\in\Ga$ arbitrarily.
Let $\alpha_i(g)=(g,1/i)\in\mathfrak A$ for $i=1,2,\ldots$. 
If the net is eventually in $\bigcap_{i=1}^nX_{\alpha_i(g)}^c$ 
for each $n\in\N$ then we deduce that $\lim_dg(x_d)=0$. 
Thus if $\net xdD$ is eventually in any intersection
$\bigcap_{i=1}^nX_{\alpha_i}^c$ this implies, by 
\eqref{a kelley s}, that $\lim_df(x_d)=0$ for all $f\in\Fa$. 
In other words, the $\kappa$-exhaustion is residual. 
This proves the converse implication for both claims.
\end{proof}
%
%
%

In the special case in which $\kappa=\aleph_0$
Theorem \ref{th kelley} simplifies considerably as
the collection $\Ga$ may be replaced, with no loss 
of generality, with a $\sigma$ convex combination 
of its elements, which is still an element of $\intr(\Fa)$.
This conclusion may be applied in the context of the 
following examples.

\begin{example}
\label{ex banach}
If $E$ is a Banach lattice, $\Fa=S_{E^*}\cap E^*_+$
(so that $\Fa=\intr\Fa$) and $X=S_E\cap E_+$ then, 
by Theorem \ref{th kelley} , $\Fa$ induces a countable 
exhaustion of $X$ if and only if there exists a strictly 
positive linear functional on $E$.
\end{example}

\begin{example}
\label{ex kelley}
Let $\Fa$ be a family of capacities on a Boolean 
algebra $X$ (i.e. each $f\in\Fa$ is an increasing 
function with values in $[0,1]$ and such that 
$f(0)=0$ and $f(1)=1$). Then $\intr(\Fa)$ consists 
of capacities as well. If $\Fa$ induces a countable 
exhaustion of $X$, this is equivalent, by Theorem 
\ref{th kelley}, to the existence of a capacity $\nu$ 
such that $\nu(x)=0$ if and only in $f(x)=0$ for all 
$f\in\Fa$. 
\end{example}

The exhaustion technique exploited above was inspired
by the approach of Kelley \cite[Theorem 4]{kelley} 
to the so-called Maharam problem for additive set 
functions on a Boolean algebra. Kelley's proof, based on 
the intersection number, has been extended by Galvin 
and Prikry \cite{galvin_prikry} and, more recently, by 
Balcar et al. \cite{balcar_jech_pazak}.

\section{Integral Representation Theorems}
\label{sec choquet}

Theorem \ref{th dom} implies a partial extension 
of the integral representations of Choquet 
\cite[Th\'eor\`eme 1]{choquet_56} and of Strassen
\cite[Theorem 1]{strassen}. For the former we make 
use of the concept of sufficient subset, a special case 
of which is the notion of boundary in the theory of 
Banach spaces, see e.g. \cite[Definition 1.1]{godefroy}. 

\begin{definition}
\label{def suff}
A subset $Z\subset X$ is sufficient for $X$ relatively 
to the family $\Ha\subset\Fun X$, in symbols $Z\geH X$, 
if the following is true:
\begin{equation}
\label{extreme}
\sum_{h\in\Ha}\delta(h)h(x)
	\le
\sup_{z\in Z}\sum_{h\in\Ha}\delta(h)h(z)
\qquad
x\in X,\ 
\delta\in\Delta(\Ha).
\end{equation}
\end{definition}

We designate with the symbol $\tau(\Ha)$ the initial 
topology induced by $\Ha$ on $X$. If $\Ha$ consists 
of bounded functions, it also induces a topology on 
the space $ba(X)$ (considered as the dual space of 
the set of bounded functions), necessarily weaker 
than the corresponding weak$^*$ topology. We 
denote the latter topology by $w^*(\Ha)$.

\begin{theorem}
\label{th choquet}
Let $\Ha\subset\Fun X$ consist of bounded functions.
Let $\mathscr Z$ be a family of subsets of $X$, linearly 
ordered by inclusion. Then $Z\geH X$ for every 
$Z\in\mathscr Z$ if and only if for each $x\in X$ 
there exists $m_x\in\Print X\Ha$ such that
$m_x(Z^c)=0$ for all $Z\in\mathscr Z$ and
\begin{equation}
\label{choquet}
h(x)
	\le
\int h(y)m_x(dy),
\qquad
h\in\Ha.
\end{equation}
\end{theorem}

\begin{proof}
Fix $x\in X$ and $Z\in\mathscr Z$ and apply Theorem 
\ref{th dom} upon replacing $X$ with $\Ha$, $\Ga$ 
with $\{e_x\}$ and $\Fa$ with $\Ev Z$. We obtain 
$m'_x\in\Print Z\Ha$ satisfying \eqref{choquet}. 
Define $m_x(E)=m'_x(E\cap Z)$ for all $E\subset X$.
The set 
\begin{equation}
\label{Rx}
\Ring_x(\Ha;Z)
=
\{m\in\Prob X:m(Z^c)=0
\text{ and }
m\text{ satisfies }\eqref{choquet}\}
\end{equation}
is thus a non empty, $w^*(\Ha)$ compact set. Moreover, 
for given $Z,Z'\in\mathscr Z$ the inclusion $Z\subset Z'$ 
implies 
$\Ring_x(\Ha;Z)
	\subset
\Ring_x(\Ha;Z')$. 
But then it is enough to select for each $x\in X$ an 
element of
$\bigcap_{Z\in\mathscr Z}\Ring_x(\Ha;Z)$.
The converse implication is obvious.
\end{proof}

In other words each $x\in X$ is the sub barycentre
of some element of $\Print X\Ha$ which represents
it. If $Z$ is the set of extreme points of a convex, compact subset of $X$ of a locally 
convex space $E$ and $\Ha$ the dual of $E$ (as in the 
original formulation of Choquet \cite{choquet_56}) it 
is then obvious that $Z\geH X$ and that the measure 
$m_x$ which represents $x$ relatively to $\Ha$ may 
be chosen to be a countably additive, regular Borel 
measure. The latter property may conflict with the 
condition $m_x(Z^c)=0$ if $Z$ is not a Borel set, as 
noted by Bishop and De Leeuw \cite{bishop_deleeuw}. 
Theorem \ref{th choquet} requires instead no special 
assumption on $X$, $Z$ or $\Fa$ and raises no 
measurability issue. At the same time it does not 
permit a deeper characterization of representing 
measures through additional properties. 

It is obvious from the definition that for a symmetric
family $\Ha$ of bounded functions, a net converges 
uniformly on $X$ if and only if it converges uniformly 
on any of its sufficient subsets. This suggests that 
sufficient subsets may induce a topological characterization
of $\Ha$. Indeed several authors have exploited Choquet 
integral representation to deduce compactness criteria,
including Rainwater \cite[p. 25]{phelps}, Godefroy
\cite[Theorem I.2]{godefroy}, Bourgain and Talagrand 
\cite[Th\'eor\`eme 1]{bourgain_talagrand} and, 
recently, Pfitzner \cite{pfitzner}. All of these results
assume linearity of $X$ and of $\Ha$. In our general 
setting similar conclusions may be reached but under 
additional assumptions on the set of representing
measures
$\mathscr R(\Ha;Z)
=
\bigcup_{x\in X}\mathscr R_x(\Ha;Z)$ .

\begin{corollary}
\label{cor choquet}
Let $\Ha\subset\Fun X$ be a symmetric set of
bounded functions that separate the points of 
$X$ and let $Z\geH X$ be $\tau(\Ha)$ closed.
Assume the existence of a $w^*(\Ha)$ closed 
subset
$\mathcal R
	\subset
\mathscr R(\Ha;Z)$
such that
\begin{equation}
\mathcal R
\cap
\mathscr R_x(\Ha;Z)\ne\emp
\qquad
x\in X.
\end{equation}
Then, a sequence in $\Ha$ converges pointwise on 
$X$ if and only if it converges pointwise on $Z$.
\end{corollary}

\begin{proof}
Of course $\Prob X$ is $w^*(\Ha)$ compact and,
under the present assumptions, so is $\mathcal R$. 
Let $\neta x$ be a net in $X$ and choose 
$m_a
	\in
\mathcal R\cap \Ring_{x_a}(\Ha;Z)$. 
By assumption, the corresponding net $\neta m$ 
admits a cluster point $m\in\mathcal R$ and an
element $x\in X$ such that $m\in\mathscr R_x(\Ha;Z)$. 
But then, along some subnet $\net x\beta{\mathfrak B}$
\begin{align*}
h(x)
	=
\int_Z h(z)m(dz)
	=
\lim_\beta\int_Z h(z)m_\beta(dz)
	=
\lim_\beta h(x_\beta)
\qquad
h\in\Ha.
\end{align*}
This implies that $(X,\tau(\Ha))$ is compact and 
Hausdorff, and that so is $(Z,\tau(\Ha))$. The
restriction $\hat h$ of each $h\in\Ha$ to $Z$ is
of course continuous and, by \eqref{choquet},
the value $h(x)$ of $h$ at $x$ may be viewed as 
the action $\phi_x(\hat h)$ on $\hat h$ of a 
continuous linear functional $\phi_x$ defined
on the class of all continuous functions over the
topological space $(Z,\tau(\Ha))$. Consequently
we may write $h(x)=\int h(z)\mu_x(dz)$ with 
$\mu_x$ a countably additive, regular probability 
measure on the Borel subsets of $Z$. But then 
pointwise convergence on $Z$ implies pointwise 
convergence on $X$ as a simple consequence of 
bounded convergence. 
\end{proof}

If $X$ is $\tau(\Ha)$ compact and each $h\in\Ha$ 
is linear the condition of Corollary \ref{cor choquet}
holds true.

Specifying a linear structure induces further integral
representation results, near to the original findings
of Strassen \cite{strassen} and of Cartier et al. 
\cite{cartier_fell_meyer}.

\begin{lemma}
\label{lemma strassen}
Let $Z$ be a countable and symmetric subset of a 
topological vector space $X$ and $\Fa\subset\Fun X$ 
a pointwise bounded family of sublinear functionals. 
A linear functional $\varphi$ on $X$ which satisfies 
the condition
\begin{equation}
\label{finite w}
\big(\forall z\in Z\big)
\big(\forall E\subset\Fa,\text{ finite}\big)
\big(\exists f\in E^c\big):\ 
\varphi(z)\le f(z)
\end{equation}
admits the representation
\begin{equation}
\label{strassen w}
\varphi(y)
	=
\int_\Fa T(f,y)m(df),
\qquad
y\in\lin(Z)
\end{equation} 
in which
(a)
$m\in\Prob\Fa$ vanishes on finite sets and 
(b) 
$T\in\Fun{\Fa\times X}$ is such that $T(f,\cdot)$ 
is a linear functional $\le f$ for each $f\in\Fa$. 
If $X$ is an $F$- space and each 
$f\in\Fa$ is continuous, then \eqref{strassen w} 
extends to $\cl\lin(Z)$.
\end{lemma}

\begin{proof}
We easily obtain from Hahn-Banach a family $\{\chi_f:f\in\Fa\}$ 
of linear functionals on $X$, each of which satisfying 
the inequality $\chi_f\le f$. Let $z_1,z_2,\ldots$ be 
an enumeration of $Z$. Proceeding recursively, for 
each $n\in\N$ we can, by \eqref{finite w}, choose 
$f_n\in\Fa\setminus\{f_1,\ldots,f_{n-1}\}$ such that 
$\varphi(z_n)\le f_n(z_n)$ and, using again Hahn 
Banach, obtain a linear functional $t_n\le f_n$ 
defined on $X$ and such that $t_n(z_n)=f_n(z_n)$. 
Define $T(f,z)$ implicitly by letting 
\begin{equation}
T(f_n,z)=t_n(z),\ 
n=1,2,\ldots
\qtext{or else}
T(f,z)=\chi_f(z),
\quad
f\notin\{f_1,f_2,\ldots\}.
\end{equation}
By construction,
\begin{align*}
T(f,\cdot)\le f,
\qquad
f\in\Fa
\qqand
\varphi(y)
	\le
\inf_{\{E\subset\Fa:\text{ finite}\}}\sup_{f\notin E}T(f,y),
\qquad
y\in\lin(Z).
\end{align*}
As an immediate consequence of Theorem \ref{th dom}
and symmetry of $Z$ we deduce that for each finite 
$E\subset\Fa$ there exists $\widebar m_E\in\Prob{E^c}$ 
such that the probability $m_E\in\Prob\Fa$, defined 
by letting 
$m_E(A)
	=
\widebar m_E(A\cap E^c)$,
satisfies \eqref{strassen w}. The family of weak$^*$ closed 
subsets of $\Prob\Fa$ obtained by letting $E$ range over 
all finite subsets of $\Fa$ has the finite intersection property 
so that the claim follows. This establishes the first claim. If
$X$ is an $F$-space, then $\Fa$ 
is uniformly bounded and, by the inequality $T(f,\cdot)\le f$, 
so is the family $T(f,\cdot)$ for $f\in\Fa$. The last claim 
follows from uniform convergence.
\end{proof}

Condition \eqref{finite w} is satisfied e.g. if $\F$ is an infinite,
$X$-separable set and if
$\varphi(z)<\sup_{f\in\Fa}f(z)$
for each $z\in Z$.

While Lemma \ref{lemma strassen} does not use any
form of measurability, if we introduce some topological
assumptions we obtain a representation similar to that 
of Strassen \cite[Theorem 1]{strassen}. The main d
ifference is that in our formulation it is not assumed
the existence of an {\it a priori} given probability space 
on $\Fa$. We denote by $\cl[X]\Ha$ the $X$-closure of 
$\Ha\subset\Fun X$ and by $\Bor_X(\Ha)$ the 
$\sigma$ algebra generated by the $X$- open subsets 
of $\Ha$.

\begin{theorem}[Strassen]
\label{th strassen}
Let $X$ be a real vector space and $\Fa\subset\Fun X$ 
a pointwise bounded family of sublinear functionals. 
A linear functional $\varphi$ on $X$ satisfies the 
condition
\begin{equation}
\label{strassen dom}
\varphi(x)\le\sup_{f\in\Fa}f(x),
\qquad
x\in X
\end{equation}
if and only if it may be represented in the form
\begin{equation}
\label{strassen}
\varphi(x)
	=
\int T(x,f)\lambda(df),
\qquad
x\in X
\end{equation}
in which
\begin{enumerate}[(a).]
\item
$\lambda$ is a Radon probability on
$\Bor_X\big(\cl[X]\Fa\big)$;
\item
$T(x,\cdot)$ is $\Bor_X\big(\cl[X]\Fa\big)$ measurable
for each $x\in X$;
\item
for each $x,y,z\in X$ and $a,b\in\R$ there exists
a $\lambda$ null set $N(a,b;y,z)\in\Bor_X(\cl[X]\Fa)$ 
such that
\begin{equation}
\label{HB}
T(x,f)\le f(x)
\qand
T(ay+bz,f)
	=
aT(y,f)+bT(z,f)
\qquad
f\notin N(a,b;y,z).
\end{equation}
\end{enumerate}
Moreover, if $X$ is a separable topological vector 
space, $\Fa$ is $X$-closed and $\varphi$ and each 
$f\in\Fa$ are continuous, then one may choose $T$ 
such that \eqref{HB} holds for all $f$ outside some
fixed $\lambda$ null set and, if $X$ is an $F$-space, 
even for all $f\in\Fa$.
\end{theorem}

\begin{proof}
Given that \eqref{strassen dom} remains unchanged
if we replace $\Fa$ with its $X$-closure, we can assume
with no loss of generality that $\Fa$ is $X$-closed and 
thus $X$-compact as well as Hausdorff. By Theorem 
\ref{th dom} we can write
$\varphi(x)
	\le
\int_\Fa f(x)m(df)$ 
for some $m\in\Print\Fa{\Ev X}$ and all $x\in X$. 
Evaluators are continuous functions of $\Fa$ if 
the latter set is given the $X$-topology and therefore
by the Riesz-Markoff representation Theorem, we may 
replace $m$ with a regular Borel (and thus Radon)
probability $\lambda\in\Prob{\Bor_X(\Fa)}$. The rest 
of the proof is very similar to the original proof of
Strassen. If $L$ is the vector subspace of 
$\Fun{\Fa,X}$ spanned by elements of the form 
$x\set E$, with $E\in\Bor_X(\Fa)$, then $\varphi$ 
admits a linear extension $\widetilde\varphi$ to 
$L$ satisfying
\begin{equation}
\label{simple}
\widetilde\varphi(h)
	\le
\int f\big(h(f)\big)\lambda(df),
\qquad
h\in L.
\end{equation}
This follows from the Hahn-Banach Theorem once
observed that the right hand side of \eqref{simple}
is sublinear on $L$. Write 
$\mu_x(E)
	=
\widetilde\varphi(x\set E)$. 
Given that $\mu_x$ is additive and that $\mu_x\ll\lambda$ 
we conclude that $\mu_x$ is itself a regular Borel measure 
on $\Bor_X(\Fa)$ admitting a Radon-Nikodym derivative 
denoted by $T_x$. Write $T(x,f)=T_x(f)$. 
Properties \tiref a and \tiref b are clear; \tiref c follows
from the linearity of $\widetilde\varphi$. 
Sufficiency is also clear since, by \tiref c,
\begin{align*}
\varphi(x)
	=
\int T(x,f)\lambda(df)
	\le
\int f(x)\lambda(df)
	\le
\sup_{f\in\Fa}f(x).
\end{align*}

Assume now that $X$ is a separable topological 
vector space and denote by $X_0$ the countable, 
rational vector subspace of $X$ which is dense in 
$X$. For each $f$ outside of the null set
$N
=
\bigcup_{a,b\in\Q,\ y,z\in X_0}N(a,b;y,z)$
$T(\cdot,f)$ is a linear functional on $X_0$ and 
$\le f$ (and thus continuous). Consider the extension
$T'(\cdot,f)$ of $T(\cdot,f)$ to the whole of $X$ 
obtained by continuity. Let $U(x,f)=T'(x,f)\set{N^c}(f)$.
It is obvious that $U$ satisfies properties \tiref{a}--%
\tiref{c} for each $f\notin N$. At the same time, if 
$\seqn x$ is a sequence in $X_0$ converging 
to $x$ we have
\begin{align}
\varphi(x)
	=
\lim_n\varphi(x_n)
	=
\lim_n\int_{N^c} T(x_n,f)\lambda(df)
	=
\int U(x,f)\lambda(df)
\end{align}
by bounded convergence.

Assume eventually that $X$ is an $F$ space and 
consider the set $\Psi$ of sublinear functionals 
satisfying
\begin{equation}
\psi(x)
	\le
\sup_{f\in\Fa}f(x),
\qquad
x\in X.
\end{equation}
$\Psi$ is of course pointwise bounded, $X$-closed and 
each $\psi\in\Psi$ is continuous, by uniform boundedness. 
If $X$ is separable the $X$-topology on $\Psi$ is metrizable, 
\cite[s. 307, p. 267 ]{tkachuk}. Consider covering the 
closure of the above set $N$ with finitely many balls of 
radius $2^{-k}$ with centres $h^k_1,\ldots,h^k_I$ and, 
for each $i=1,\ldots,I$, let $\chi_i^k$ be a linear functional
$\le h^k_i$. Let $E_1^k,\ldots,E^k_I$ be the disjoint
collection obtained by the cover above and define
\begin{equation}
V^k(x,f)
	=
\sum_{i=1}^I\chi^k_i(x)\set{E_i^k}(f),
\qquad
f\in\widebar N,\ 
x\in X.
\end{equation}
Of course $V^k(x,\cdot)$ is $\Bor_X(\Fa)$ measurable, 
$V^k(\cdot,f)$ is a linear functional in $\Psi$ for each 
$f\in\widebar N$ and 
$V^k(\cdot,f)
	\le 
h^k_f$ 
for some $h^k_f\in\Fa$ such that 
$\norm{h^k_f-f}
	\le
2^{-k}$. 
Because $\Psi$ is metrizable we can extract 
a subsequence (still indexed by $k$) which 
$X$-converges in $\Psi$ to a linear limit 
$V(\cdot,f)$. Observe that for fixed 
$x\in X$ we have 
$V(x,f)
	=
\lim_kV^k(x,f)
\le 
f(x)+\lim_n(h^k_f-f)(x)
=
f(x)$. 
Eventually, for each $x\in X$ the function $V(x,\cdot)$ 
is $\Bor_X(\Fa)$ measurable, since the pointwise limit
of measurable functions. The proof is then complete
upon replacing $T$ in \eqref{strassen} with
$U(x,f)+V(x,f)\set N(f)$.
\end{proof}

\section{Summable Functions}
\label{sec sum}
We introduce the following family of functions:

\begin{definition}
A function $g\in\Fun X$ is said to be summable along
$\Fa\subset\Fun X$, in symbols $g\in\mathcal S_\Fa(X)$, 
if the series $\sum_ng(x_n)a(x_n)$ converges for every 
sequence $\seqn x$ in $X$ and every $a\in\Fun{X}$ 
such that
\begin{equation}
\label{F sum}
\sup_{f\in\Fa}\sum_n\abs{f(x_n)a(x_n)}
	<
+\infty.
\end{equation}
\end{definition}

\begin{corollary}
\label{cor sum}
Let $\Fa\subset\Fun X$ be pointwise bounded.
Then, $g\in\mathcal S_\Fa(X)$ if and only if 
\begin{equation}
\label{F dom}
\abs{g(x)}
	\le
C_g\int_\Fa\abs{ f(x)}m_g(df)
\qquad
x\in X
\end{equation}
for some $C_g\ge0$ and $m_g\in\Print\F{\Ev X}$. If 
$\Fa$ is $X$-closed then $m_g$ may be chosen to be 
a regular, Borel probability on $\Bor_X(\Fa)$.
\end{corollary}

\begin{proof}
Assume that $g\in\mathcal S_\Fa(X)$ and let $\seqn x$ 
and $a\in\Fun X$ satisfy \eqref{F sum}. The series 
$\sum_ng(x_n)a(x_n)$ converges absolutely. Let 
$h_1,h_2,\ldots\in\Fun[0]{X,\R_+}$ be such that
\begin{equation}
\label{bb}
\sup_{f\in\Fa}\br {\abs f}{h_k}
	\le
2^{-k}
\qquad
k\in\N.
\end{equation}
Then,
\begin{equation}
+\infty
	>
\sup_{f\in\Fa}\sum_k\br{\abs f}{h_k}
	=
\sup_{f\in\Fa}\sum_x\abs{f(x)}\sum_kh_k(x)
	=
\sup_{f\in\Fa}\sum_n\abs{f(x_n)a(x_n)}
\end{equation}
where $x_1,x_2,\ldots$ is an enumeration of the countable 
set $\bigcup_k\{h_k>0\}$ and $a\in\Fun{X,\R_+}$ is defined
via
\begin{equation}
a(x)=\sum_kh_k(x)
\qtext{if}
\sup_{f\in\Fa}\abs{f(x)}+\abs{g(x)}>0
\qtext{or else}
a(x)=0.
\end{equation}
By assumption,
$+\infty
	>
\sum_n\abs{g(x_n)a(x_n)}
	=
\sum_k\br{\abs{g}}{h_k}$
and therefore $\lim_k\br{\abs{g}}{h_k}=0$. Since 
every sequence $\seq hk$  in $\Fun[0]{X,\R_+}$ for which
$\lim_k\sup_{f\in\Fa}\br{\abs f}{h_k}=0$ admits a 
subsequence satisfying \eqref{bb}, we conclude that
\begin{equation}
\label{cts}
\lim_k\sup_{f\in\Fa}\br{\abs f}{h_k}=0
\qtext{implies}
\lim_k\br{\abs{g}}{h_k}=0.
\end{equation}
Observing that the function $\br\cdot\cdot$ is separately 
homogeneous, we deduce that the inclusion 
$g\in\mathcal S_\Fa(X)$ implies the 
existence of $C_g>0$ such that 
$\br{\abs{g}}\delta
\le 
C_g\sup_{f\in\Fa}\br{\abs f}\delta$
for each $\delta\in\Delta(X)$ so that \eqref{F dom} 
follows from Theorem \ref{th dom}.

Conversely, if \eqref{F dom} holds, and if $\seqn x$ 
and $a\in\Fun X$ satisfy \eqref{F sum}, then 
\begin{align*}
+\infty
	>
C_g\sup_{f\in\Fa}\sum_n\abs{f(x_n)a(x_n)}
	\ge
C_g\sum_n\int_{\Fa}\abs{f(x_n)a(x_n)}m(df)
	\ge
\sum_n\abs{g(x_n)a(x_n)}.
\end{align*}
The last claim is an obvious consequence of well
known results once noted that $e_x$ is an $X$-continuous
function on $\Fa$ and that $\Fa$ is $X$-compact
by virtue of Tychonoff theorem.
\end{proof}

Corollary \ref{cor sum} is a fully non linear extension
of a well known result of Grothendieck-Pietsch 
\cite[p. 60]{diestel} which concerns $p$-summing 
operators with $p\ge1$, i.e. bounded linear operators 
$T\in\Fun{X,Y}$ ($Y$ a Banach space) which satisfy the
condition
\begin{equation}
\sum_n\norm{Tx_n}^p<\infty
\qtext{whenever}
\sum_n\abs{x^*x_n}^p<\infty
\qquad
x^*\in S_{X^*}.
\end{equation}
This criterion may be equivalently formulated as the
condition
\begin{equation}
\lim_k\sum_{x\in X}\norm{Tx}^ph_k(x)
	=
0
\qtext{whenever}
\lim_k\sup_{x^*\in S_{X^*}}\sum_{x\in X}\abs{x^*x}^ph_k(x)
	=
0
\end{equation}
for every sequence $\seq hk$ in $\Fun[0]{X,\R_+}$, see 
\cite[p. 59]{diestel}, which corresponds to the inclusion 
$g\in\mathcal S_\Fa(X)$ when $\Fa$ consists of element
of the form $f(x)=\abs{x^*x}^p$ for some $x^*\in S_{X^*}$ 
and $g(x)=\norm{Tx}^p$. Condition \eqref{F dom} is then 
a restatement of the inequality of Grothendieck and Pietsch.

Corollary \ref{cor sum} rests on an implicit Banach
space structure which is worth making explicit. Assume
that $\Fa$ is pointwise bounded and, with no loss
of generality, that $\sup_{f\in\Fa}\abs{f(x)}>0$ for
al $x\in X$. The space
\begin{equation}
\label{ell F X}
\ell_\Fa(X)
	=
\Big\{
h\in\Fun X:
\sup_{f\in\Fa}\sum_x\babs{f(x)h(x)}<+\infty
\Big\}
\end{equation}
contains then $\Fun[0] X$. Endowed with pointwise 
order and with the norm
\begin{equation}
\label{norm}
\norm h
	=
\sup_{f\in\Fa}\sum_x\babs{f(x)h(x)},
\end{equation}
$\ell_\Fa(X)$ becomes a Banach lattice on which the 
bilinear form
\begin{equation}
\label{inner}
\inner fh
	=
\sum_xf(x)h(x)
\qquad
f\in\Fa,\ h\in\ell_\Fa(X)
\end{equation}
permits to associate with each $f\in\Fa$ an element
of $S_{\ell_\Fa(X)^*}$. 

\begin{corollary}
\label{cor order cts}
Let $\Fa\subset\Fun X$ be a pointwise bounded set 
satisfying $\sup_{f\in\Fa}\abs{f(x)}>0$ for each 
$x\in X$. If $\varphi\in\ell_\Fa(X)^*$ then the 
associated function $T\varphi$ defined as
$T\varphi(x)
=
\varphi(\delta_x)$
belongs to $\mathcal S_\Fa(X)$. In addition, 
\begin{equation}
\label{order cts}
\varphi(h)
	=
\sum_{x\in X}T\varphi(x)h(x)
\qquad
h\in\ell_\Fa(X)
\end{equation}
if and only if $\varphi$ is order continuous
(in symbols $\varphi\in\ell_\Fa(X)^o)$.
\end{corollary}

\begin{proof}
Let $\varphi\in\ell_\Fa(X)^*$, fix $h\in\ell_\Fa(X)$
and define $h_\alpha\in\ell_\Fa(X)$ as the restriction
of $h$ to some finite subset $X_\alpha$ of $X$. Then,
\begin{align}
\sum_{x\in X_\alpha}\abs{T\varphi(x)h(x)}
	=
\abs{\varphi\big(h_\alpha\sgn(T\varphi)\big)}
	\le
\norm\varphi\norm{h_\alpha}
	\le
\norm\varphi\norm h
\end{align}
we conclude that $T\varphi\in\mathcal S_\Fa(X)$. If $\varphi$
satisfies \eqref{order cts} it is clearly order continuous.
If, conversely, $\varphi$ is order continuous  then
the net $\neta h$ (with $\mathfrak A$ being directed
by inclusion of the finite subsets of $X$) is order 
convergent to $h$ so that
$\varphi(h)
	=
\lim_\alpha\varphi(h_\alpha)
	=
\lim_\alpha\sum_{x\in X_\alpha}T\varphi(x)h(x)
	=
\sum_{x\in X}T\varphi(x)h(x)$.
\end{proof}

The map $T$ defined in Corollary \ref{cor order cts} 
thus establishes a linear isomorphism between 
$\ell_\Fa(X)^o$ and $\mathcal S_\Fa(X)$.

\BIB{abbrv}

\end{document}